\documentclass{article}
	\usepackage{amsmath,amssymb,amsfonts,amsthm,hyperref,listings}
\usepackage[english]{babel}
\usepackage{enumitem}

\theoremstyle{definition}
\newtheorem{definition}{Definition}[section]
\newtheorem{example}[definition]{Example}

\theoremstyle{remark}

\theoremstyle{plain}
\newtheorem{lemma}[definition]{Lemma}
\newtheorem{proposition}[definition]{Proposition}
\newtheorem{corollary}[definition]{Corollary}
\newtheorem{theorem}[definition]{Theorem}

\title{A Method for Solving quadratic Equations
in Real Quaternion Algebra by using Scilab software}

\author{
Geanina ZAHARIA
\\  Diana-Rodica MUNTEANU }

\begin{document}
\maketitle
\begin{abstract}

	In this paper, we present some numerical applications for the equation $x^2+ax+b=0$, where $a, b$ are two quaternionic elements in $\mathbb{H}(\alpha,\beta)$.
	$\mathbb{H}(\alpha,\beta)$ represents the algebra of real quaternions with parameterized  coefficients by $\alpha$ and $\beta$. The algebra of real quaternions is an extension of complex numbers and is represented by algebraic objects called quaternions. These quaternions are composed of four components: a real part and three imaginary components. In general, $\mathbb{H}(\alpha,\beta)$ indicates a family of parameterized quaternion algebras, in which the specific values of $\alpha$ and $\beta$ determine the specific properties and structure of the quaternion algebra.
	Based on well-known solving methods, we have developed a new numerical algorithm that solves the equation for any quaternions a and b in any algebra $\mathbb{H}(\alpha,\beta)$.

\end{abstract}


\section{Introduction}

	Quaternions are a number system first introduced in 1843 by Irish mathematician Sir William Rowan Hamilton. Hamilton was seeking a way to extend the complex numbers to three dimensions and realized that he could do so by adding an additional imaginary unit. 

	Quaternions are different from complex numbers in that they are non-commu- tative. Quaternions have found many practical applications in fields such as computer graphics, physics, and engineering. For instance, they are used in computer graphics to represent 3D rotations and orientations, and in aerospace engineering to model spacecraft altitude and control systems.

	Quaternions are essential in control systems for guiding aircraft and rockets: each quaternion has an axis indicating the direction and a magnitude determining the size of the rotation. Instead of representing an orientation change through three separate rotations, quaternions use a single rotation to achieve the same transformation.

Despite their usefulness, quaternions are not as widely used as complex numbers, largely due to their non-commutative nature. However, they remain an important topic in mathematics and physics, and continue to be studied and applied in various fields to this day. ([AM; 09], [FS; 15], [FSH; 15], [HS; 02], [PR; 97])

	We will numerically solve the monic quadratic equation with quaternion coefficients in the algebra $\mathbb{H}(\alpha,\beta)$ using Scilab, a free and open-source software for numerical computation.

	We chose to use the Scilab software to numerically solve the monic quadratic equations with quaternionic coefficients in the algebra $\mathbb{H}(\alpha,\beta)$ because Scilab is a free and open-source software, making it accessible and usable by a large number of users. Additionally, this software allows us to customize and adapt it to the specific needs and requirements of our problem. Scilab is renowned for its powerful functionality in numerical computation. It offers a wide range of mathematical and algebraic functions, including an integrated solver for polynomial equations. The built-in polynomial equation solver in Scilab provides us with the necessary tools to efficiently solve the monic quadratic equation with quaternionic coefficients. Scilab, such as Matlab, which is more widely known, has a user-friendly and intuitive interface, facilitating ease of use and navigation within the software. The programming is very intuitive and dosen’t require definition of any parameters, so the main
focus remains the mathematical modeling of the equations and the algorithm. This decision allows us to obtain precise and efficient results in studying and applying our new findings in quaternion algebra.

 	The aim of the paper is to present an innovative, efficient, and accurate method for the numerical solution of monic quadratic equations in the algebra of real quaternions using the Scilab software. We develop a new algorithm that solves these equations for any quaternionic coefficients in any algebra $\mathbb{H}(\alpha,\beta)$. Our ultimate goal is to contribute to the development and application of this knowledge in various fields such as computer graphics, physics, and engineering, opening up new research and application perspectives for quaternions and monic quadratic equations with quaternionic coefficients.

\section{PRELIMINARIES}

	The quadratic equation has been explored in the context of Hamilton quaternions in the works [HT; 02], [PR; 97]. In [HT; 02] analyzed the equation $x^2+bx+c=0$ and obtained explicit formulas for its roots. These formulas were subsequently used in the classification of quaternionic Möbius transformations [PS; 09], [CP; 04].
	In Hamilton quaternions, every nonzero element can be inverted, while in H(alfa, beta)  there exist split quaternions that cannot be inverted.
	 In an algebraic system, finding the roots of a quadratic equation is always connected to the factorization of a quadratic polynomial [LSS; 19]. In the case of real numbers ($\mathbb{R}$) and complex numbers ($\mathbb{C}$), the two problems are identical. However, in noncommutative algebra, these two problems are interconnected. Scharler et al. [SSS; 20] analyzed the factorizability of a quadratic split quaternion polynomial, revealing certain information about the roots of a split quaternionic quadratic equation.
	
	In a publication from 2022 [FZ; 22] exploring algebras derived from the Cayley-Dickson process presents challenges in achieving desirable properties due to computational complexities. Hence, the discovery of identities within these algebras assumes significance, aiding in the acquisition of new properties and facilitating calculations. To this end, the study introduces several fresh identities and properties within the algebras derived from the Cayley-Dickson process. Additionally, when specific elements serve as coefficients, quadratic equations in the real division quaternion algebra can be solved, showcasing the authors ability to provide direct solutions without relying on specialized software.
	
	In the paper [CW; 22], the author specifically focuses on deriving explicit formulas for the roots of the quadratic equation $x^2+bx+c=0$ where $b$ and $c$ are split quaternions  ($\mathbb{H}_S$).
	
	The same subject can be found in [AM; 09], where quadratic formulas for generalized quaternions are studied. It focuses on obtaining explicit formulas for the roots of quadratic equations in this specific context of generalized quaternions.\medskip

Let $\mathbb{H}(\alpha,\beta)$ be the generalized quaternion algebra over an arbitrary field $\mathbb{K}$, that is the algebra of the elements of the form 
$q = q_1 + q_2 e_1 + q_3 e_2 + q_4 e_3 $ where $q_{i}\in \mathbb{K},$ $ i\in \{1,2,3,4\}$, and the basis elements $\{1,e_{1},e_{2},e_{3}\}$ satisfy the
following multiplication table: 

\medskip\ \vspace{3mm}
\begin{equation}\label{eq1}
\begin{tabular}{c|cccc}
$\cdot $ & $1$ & $e_{1}$ & $e_{2}$ & $e_{3}$ \\ \hline
$1$ & $1$ & $e_{1}$ & $e_{2}$ & $e_{3}$ \\ 
$e_{1}$ & $e_{1}$ & $\alpha$ & $e_{3}$ & $\alpha e_{2}$ \\ 
$e_{2}$ & $e_{2}$ & $-e_{3}$ & $\beta$ & $- \beta e_{1}$ \\ 
$e_{3}$ & $e_{3}$ & $- \alpha e_{2}$ & $\beta e_{1}$ & $-\alpha \beta$
\medskip
\end{tabular}
\end{equation}

	The conjugate of a quaternion is obtained by changing the sign of the imaginary part: $\overline{q} = q_1 - q_2 e_1 - q_3 e_2 - q_4 e_3$, where $q = q_1 + q_2 e_1 + q_3 e_2 + q_4 e_3$.

	The norm of a quaternion is defined as the sum of the squares of its components, for this case, the norm is: 
 \begin{equation*}
\boldsymbol{n}\left( q\right)=q \cdot \overline{q} = ||q||^2= q_{1}^{2}- \alpha q_{2}^{2} - \beta q_{3}^{2}+ \alpha \beta q_{4}^{2}.
\end{equation*}

If for $x \in \mathbb{H}\left(\alpha,\beta \right) $, the relation $n(x)=0$ implies $x=0$, then the algebra $\mathbb{H}\left(\alpha,\beta \right) $ is called a division algebra, othewise the quaternion algebra is called a split algebra.(see [FS;15])

If $\alpha$ and $\beta$ are negative real numbers, it becomes a division algebra, therefore the norm will be different from zero.
 	The role of $\alpha$ and $\beta$ is to parameterize the coefficients of the quaternion algebra $\mathbb{H}(\alpha,\beta)$. These values determine the specific properties and structure of the quaternion algebra. In the multiplication table given in Equation \eqref{eq1}, $\alpha$ and $\beta$ appear as parameters that determine the specific structure and properties of the quaternion algebra $\mathbb{H}(\alpha,\beta)$.\\
	The role of the norm is to provide a measure of the size of a quaternion in the algebra $\mathbb{H}(\alpha,\beta)$. The norm expression involves the coefficients $q_1, q_2, q_3, q_4$, and the parameters $\alpha$ and $\beta$. The norm plays a crucial role in determining whether the algebra $\mathbb{H}(\alpha,\beta)$ is a division algebra or a split algebra, based on whether the norm is nonzero or zero, respectively.\\

Split quaternions form an algebraic structure and are linear combinations with real coefficients. Every quaternion can be written as a linear combination of the elements $1$, $e_1$, $e_2$, and $e_3$, where $e_1$, $e_2$, and $e_3$ are the imaginary units that satisfy the relations $e_1^2 = \alpha$, $e_2^2 = \beta$, and $e_3^2 = -\alpha\beta$.\medskip\\

We will now present some of the most important properties and relations of quaternions, which play a fundamental role in various fields such as physics, engineering, computer science, and applied mathematics:

\begin{itemize}
\item The addition is done component-wise:\\
$a=a_{1}\cdot 1+a_{2}e_{1}+a_{3}e_{2}+a_{4}e_{3},$\\
$b=b_{1}\cdot 1+b_{2}e_{1}+b_{3}e_{2}+b_{4}e_{3},$\\
$ \Rightarrow a+b=(a_{1}+b_{1})\cdot 1+(a_{2}+b_{2})e_{1}+(a_{3}+b_{3})e_{2}+(a_{4}+b_{4})e_{3}.$

\item Quaternion multiplication is not commutative:\\ 
$a\cdot b= (a_1 b_1+ \alpha a_2 b_2 + \beta a_3 b_3 - \alpha \beta a_4 b_4) + e_1 (a_1 b_2+  a_2 b_1 - \beta a_3 b_4 + \beta a_4 b_3) + e_2 (a_1 b_3+ \alpha a_2 b_4 +  a_3 b_1 - \alpha a_4 b_2) + e_3 (a_1 b_4+  a_2 b_3 -  a_3 b_2 + a_4 b_1)$\\
$b\cdot a= (a_1 b_1+ \alpha a_2 b_2 + \beta a_3 b_3 - \alpha \beta a_4 b_4) + e_1 (a_2 b_1+  a_1 b_2 - \beta a_4 b_3 + \beta a_3 b_4) + e_2 (a_3 b_1+ \alpha a_4 b_2 +  a_1 b_3 - \alpha a_2 b_4) + e_3 (a_4 b_1+  a_3 b_2 -  a_2 b_3 + a_1 b_4)$\\ 
$\Rightarrow a \cdot b \neq b \cdot a.$

\item Quaternions are associative:
$(a\cdot b)\cdot c =a \cdot (b \cdot c)= a \cdot b \cdot c.$

\item The trace of the element q: 
\begin{equation*}
t(q)=q+ \overline{q}, 
\end{equation*}

\item The multiplication of a quaternion by a scalar:
\begin{equation*}
\alpha \cdot q = \alpha \cdot (q_1 + q_2 e_1 + q_3 e_2 + q_4 e_3) = (\alpha \cdot q_1) + (\alpha \cdot q_2)\cdot e_1 +(\alpha \cdot q_3)\cdot e_2 + (\alpha \cdot q_4)\cdot e_3.
\end{equation*}

\item The inverse of a non-zero quaternion $q$ is given by
\begin{equation*}
q^{-1} = \frac{\overline{q}}{||q||^2} =\frac{ q_1 - q_2 e_1 - q_3 e_2 - q_4 e_3}{q_{1}^{2}- \alpha q_{2}^{2} - \beta q_{3}^{2}+ \alpha \beta q_{4}^{2}}.
\end{equation*}

\item The dot product of two quaternions can be defined as\\
 $q \cdot r = (qr + rq)/2$.
\end{itemize}

 These are just some of the many important relations and properties of quaternions.
 All these properties make quaternions a powerful tool in mathematics and practical applications.
 \begin{equation*}
\end{equation*}

\section{KNOWN RESULTS}

	In [W1] and [W2], to find the root of the equation $f(x_t)=0$, the Newton-Raphson method relies on the Taylor series expansion of the function around the estimate $x_i$ to find a better estimate $x_{i+1}$:

\begin{equation*}
f(x_{i+1})=f(x_i)+f'(x_i)(x_{i+1}-x_i)+\mathcal O(h^2)
\end{equation*}
where $x_{i+1}$ is the estimate of the root after iteration $i+1$ and $x_i$ is the estimate at iteration $i$. $\mathcal O(h^2)$ means the order of error of the Taylor series around the point $x_i$. Assuming $f(x_{i+1})=0$ and rearranging:

\begin{equation*}
x_{i+1} \approx x_i - \frac{f(x_i)}{f'(x_i)}
\end{equation*}

The procedure is as follows. Setting an initial guess $x_0$, a tolerance $\varepsilon_s$, and a maximum number of iterations $N$:

At iteration $i$, calculate $x_i \approx x_{i-1}-\frac{f(x_{i-1})}{f'(x_{i-1})}$ and $\varepsilon_r$. If $\varepsilon_r \leq \varepsilon_s$ or if $i \geq N$, stop the procedure. Otherwise, repeat.

In [HS; 02], the authors present specific formulas to solve the monic quadratic equation $x^2+bx+c=0$ with  $b, c\in \mathbb{H}\left(\alpha,\beta \right), $ where $\alpha= -1$, $\beta =-1,$ the real division algebra, according to the multiplication table presented in \eqref{eq1}.
In the following we present the results we will use in developing our solutions, and a proof of lemma 2:

\begin{lemma} ([HS;02], Lemma 2.1)
Let $A, B, C\in \mathbf{R}$ 
with the following properties: $C \neq 0,$ 
$A<0 $  implies $ A^2<4B.$

Then the equation of order $3$:
\begin{equation}
y^3+2 A y^2+(A^2-4 B) y-C^2=0
\end{equation}
has exactly one positive solution $y$.
\end{lemma}

\begin{lemma} ([HS; 02], Lemma 2.2)
Let $A, B, C \in R$ such that:

 $B\geq 0$ and
 $A < 0$ implies $A^2<4B$
then the real system:\\
\begin{equation}
\begin{cases}
  Y^2 - (A + W^2)Y + B = 0 \\
  W^3 + (A - 2Y)W + C = 0 \\
\end{cases}
\end{equation}
\\
has at most two solutions $(W,Y)$ with $W \in \mathbf{R}$ and $Y\geq 0$ as follows:

\begin{enumerate}[label=(\roman*)]
\item $W=0$, $Y=\frac{A\pm \sqrt{A^2-4B}}{2}$ provided that $C=0$, $A^2\geq 4B$;
\item $W=\pm \sqrt{2\sqrt{B}-A}$, $Y=\sqrt{B}$ provided that $C=0$, $A^2<4B$.
\item $W=\pm
\sqrt{z}$, $Y=\frac{W^3+AW+C}{2W}$  provided that $C \neq 0$ and $z$ is the unique positive solution of the real polynomial:
\begin{center}
$z^3+2A z^2+(A^2-4B)z-C^2=0.$
\end{center}
\end{enumerate}
\end{lemma}

\begin{proof}
Let $A, B, C \in \mathbf{R}$ such that $B \ge 0$ and $A < 0 \implies A^2 < 4B$. 
\newline 
We want to show that the real system


has at most two solutions $(W,Y)$ with $W \in \mathbf{R}$ and $Y \ge 0$ as follows:

$W=0$, $Y=\frac{A\pm \sqrt{A^2-4B}}{2}$ provided that $C=0$, $A^2\geq 4B$;\newline 
$W=\pm \sqrt{2\sqrt{B}-A}, Y=\sqrt{B}$ provided that $C=0$, $A^2<4B$;\newline 
$W=\pm \sqrt{z}, Y=\frac{W^3+AW+C}{2W}$ provided that $C \neq 0$ and $z$ is the unique positive solution of the real polynomial:
\begin{center}
$z^3+2Az^2+(A^2-4B)z-C^2=0$
\end{center}
From Lemma 2.1, we know that the polynomial $z^3+2Az^2+(A^2-4B)z-C^2=0$ has exactly one positive solution $z$ when $C \neq 0$.

1. and 2. are the cases when $C = 0$. In these cases, the first equation becomes a quadratic equation in $Y$. \newline 
If $A^2 \ge 4B$, there are two real solutions for $Y$, and if $A^2 < 4B$, there is one real solution for $Y$. \newline  
Since $W=0$, these solutions correspond to the cases 1. and 2. in the lemma.

3. is the case when $C \neq 0$. In this case, we can express $Y$ as a function of $W$ using the second equation: $Y = \frac{W^3 + AW + C}{2W}$. Substituting this expression for $Y$ in the first equation, we obtain a polynomial equation in $W^2$ of degree 3. Since $z$ is the unique positive solution of this polynomial, there are two solutions for $W$: $W = \pm \sqrt{z}$. These solutions correspond to the case 3. in the lemma.

In conclusion, the real system (3) has at most two solutions $(W, Y)$ with $W \in \mathbf{R}$ and $Y \ge 0$ as described in the lemma.
\end{proof}

\begin{theorem}\label{t23} ([HS;02], Theorem 2.3)
The solution of the quadratic equation $x^2+bx+c=0$ can be obtained in the following way:
\begin{enumerate}
\item Case 1. If $b,c \in \mathbf{R}$ and $b^2<4c$ then:
\begin{equation}
x=\frac{1}{2}(-b+e \cdot e_1+f \cdot e_2+g \cdot e_3)
\end{equation}
 where $e^2+f^2+g^2=4c-b^2$ where $e,f,g \in \mathbf{R}$.
\item
Case 2. If $b,c \in \mathbf{R}$ and $b^2\geq 4c$ then:
\begin{equation}
 x=\dfrac{-b\pm\sqrt{b^2-4c}}{2}
\end{equation}
\item
Case 3.
If $b \in \mathbf{R},\ c \notin \mathbf{R}$ then:
\begin{equation}
x=\frac{-b}{2}\pm\frac{m}{2} \mp \frac{ c_1}{m } \cdot e_1 \mp\frac{c_2}{m}\cdot e_2 \mp \frac{c_3}{m} \cdot e_3
\end{equation}

where $c=c_0+c_1 \cdot e_1+c_2 \cdot e_2+c_3 \cdot e_3$, and
\begin{equation}
m=\sqrt{\dfrac{b^2-4c_0+\sqrt{(b^2-4c_0)^2+16(c_1^2+c_2^2+c_3^2)}}{2}}\end{equation}

\item
Case 4.
If $b \notin \mathbf{R}$ then:
\begin{equation}  \label{ecuatia 9}
x=\dfrac{(-Re(b))}{2}-(b^\prime+W)^{-1} (c^\prime-Y)
\end{equation}

where $b^\prime=b-Re(b)=Im(b)$,
$c^\prime=c-(Re(b)/2)(b-(Re(b))/2)$,
where $(W,Y)$ are chosen in the following way:

\end{enumerate}

\begin{enumerate}[label=(\roman*)]
\item $W=0,$ $Y=(A\pm \sqrt{A^2-4B})/2$ provided that $C=0,$  $A^2 \geq4B$
\item $W=\pm \sqrt{2\sqrt{B}-A},$ $Y=\sqrt{B}$ provided that $C=0,$ $A^2 < 4B$
\item $W=\pm \sqrt{z},$ $Y=(W^3+AW+C)/2W$ provided that $C \neq 0$ and $z$ is the unique positive solution of the equation:\\

$z^3+2Az^2+(A^2-4B)z-C^2=0$\\

where $A=\vert b^\prime\vert^2+2Re(c^\prime),$ $B=\vert c^\prime\vert^2$ and $C=2Re(\overline{b^\prime}c^\prime)$.
\end{enumerate}
\end{theorem}

\begin{corollary}([HS;02], Corollary 2.4)
The equation has an infinity of solutions if $b, c \in \mathbf{R}$ and $b^2< 4c$.
\end{corollary}

\begin{corollary}([HS;02], Corollary 2.6)
The equation has an unique solution if and only if:
\begin{enumerate} 
\item $b, c \in \mathbf{R}$ and $b^2-4c=0$
 \item $b \notin \mathbf{R}$ and $C=0=A^2-4B$
\end{enumerate}
\end{corollary}

\begin{corollary} If the quadratic equation $x^2+bx+c=0$ has real coefficients $b$ and $c$, and $b^2<4c$, then the solution of the equation can be expressed as $x=\frac{1}{2}(-b+e \cdot e_1+f \cdot e_2+g \cdot e_3)$, where $e^2+f^2+g^2=4c-b^2$ and $e,f,g \in \mathbf{R}$.
\end{corollary}

\begin{corollary} If the quadratic equation $x^2+bx+c=0$ has real coefficients $b$ and $c$, and $b^2\geq 4c$, then the solutions of the equation are $x=\dfrac{-b\pm\sqrt{b^2-4c}}{2}$.
\end{corollary}

\begin{corollary}If $b$ and $c$ are the coefficients of the quadratic equation $x^2+bx+c=0$, such that $b \notin \mathbf{R}$, then the solution of the equation can be expressed as:\\
 \begin{equation*}
 x=\dfrac{(-Re(b))}{2}-(b^\prime+W)^{-1} (c^\prime-Y),
 \end{equation*}
  where $b^\prime=b-Re(b)=Im(b)$, $c^\prime=c-(Re(b)/2)(b-(Re(b))/2)$,\\
   and $(W,Y)$ are chosen such that:\\

$W=0,$ $Y=(A\pm \sqrt{A^2-4B})/2$ if $C=0$ and $A^2 \geq4B$\\
$W=\pm \sqrt{2\sqrt{B}-A},$ $Y=\sqrt{B}$ if $C=0$ and $A^2 < 4B$\\
$W=\pm \sqrt{z},$ $Y=(W^3+AW+C)/2W$ if $C \neq 0$ and $z$ is the unique positive solution of the equation
 $z^3+2Az^2+(A^2-4B)z-C^2=0,$ where $A=\vert b^\prime\vert^2+2Re(c^\prime),$ $B=\vert c^\prime\vert^2$ and $C=2Re(\overline{b^\prime}c^\prime)$.\end{corollary}

\section{The solutions of the second-degree equation in real quaternions}
	It is important to mention that the algebra $\mathbb{H}(\alpha,\beta)$ is a mathematical construction, and its properties can vary depending on the values chosen for $\alpha$ and $\beta$. When we take negative values for $\alpha$ and $\beta$ in the algebra $\mathbb{H}(\alpha,\beta)$), it becomes a division algebra. This means that every nonzero element in the algebra can be inverted. Multiplication and inversion of elements can be performed using the specific rules of this algebra.

	Therefore, for the algebra $\mathbb{H}(\alpha,\beta)$, we will take negative values for $\alpha$ and $\beta$, thus making it a division algebra, and the norm will be nonzero. If the values of $\alpha$ and $\beta$ are positive, we no longer have a division algebra because the norm is zero.

 	Next, we will describe the solution of a monic quadratic equation in the algebra of real quaternions. This statement provides an explicit formula for finding the solutions of the equation and explains how to perform the necessary calculations. It presents the general formula for the solution of the monic quadratic equation, where the equation's coefficients are represented as real quaternions, and the solution is a linear combination of the imaginary units of the quaternions. This formula is presented in a detailed manner, specifying the values of each component of the solution in terms of the coefficients and other terms involved in the equation.
\\
\begin{proposition}

Let $b=b_0+b_1 \cdot e_1+b_2 \cdot e_2+b_3 \cdot e_3$ and\\
 $c=c_0+c_1 \cdot e_1+c_2 \cdot e_2+c_3 \cdot e_3$  where $b, c$ are two quaternionic elements in $\mathbb{H}(\alpha,\beta)$ and knowing W and Y of the Theorem \ref{t23}
%
%
%
the solution of the second degree equation $x^2+bx+c=0$ is of the form
\begin{equation} \label{ecuatia 10}
 x=x_1+x_2 e_1+x_3 e_2+x_4 e_3,
\end{equation}
where:
$$ x_1=-t-[Wc_1-Y W- b_2 c_2 \alpha - b_3 c_3\beta +b_4 c_4 \alpha \beta -t(Wt-b_2^2 \alpha  - b_3^2 \beta + b_4^2 \alpha \beta )]/m;$$
$$ x_2=(Wc_2-b_2 c_1+b_2 Y+b_3 c_4 \beta  - b_4 c_3 \beta -tb_2 (W-t))/m;$$
$$ x_3=(Wc_3-b_2 c_4 \alpha - b_3 c_1+b_3 Y + b_4 c_2 \alpha-tb_3 (W-t))/m;$$
$$ x_4=(Wc_4-b_2 c_3+b_3 c_2 + b_4 c_1 + b_4 Y- tb_4 (W-t))/m;$$
and 
\begin{equation*}
 m = \sqrt{\frac{b^2-4c_0+\sqrt{(b^2-4c_0)^2+16(c_1^2+c_2^2+c_3^2)}}{2}}.
\end{equation*}
\end{proposition}

\begin{proof}
Let
$b=b_1+b_2 \cdot e_1+b_3 \cdot e_2+b_4 \cdot e_3$ and 
$c=c_1+c_2 \cdot e_1+c_3 \cdot e_2+c_4 \cdot e_3$.
\newline
for this case, the norm is: 
\begin{equation*}
\boldsymbol{n}\left( a\right)=a \overline{a} =a_{1}^{2}- \alpha a_{2}^{2} - \beta a_{3}^{2}+ \alpha \beta a_{4}^{2}.
\end{equation*}
\newline
We compute the necessary elements for applying the theorem:
$Re(b)=b_1.$\newline
Therefore, 
$$b^\prime=b-Re(b)=Im(b)=b_2 \cdot e_1+b_3 \cdot e_2+b_4 \cdot e_3 $$
and\\
$$c^\prime=c-(Re(b)/2)(b-(Re(b))/2)$$\\
$$=c_1+c_2 \cdot e_1+c_3 \cdot e_2+c_4 \cdot e_3-  \frac{b_1}{2}\left(b_1+b_2 \cdot e_1+b_3 \cdot e_2+b_4 \cdot e_3 - \frac{b_1}{2} \right) $$\\
$$=\left(c_1-\frac{b^2_1}{2}+\frac{b^2_1}{4}\right)+\left(c_2-\frac{b_1b_2}{2}\right)e_1+\left(c_3-\frac{b_1b_3}{2}\right)e_2+\left(c_4-\frac{b_1b_4}{2}\right)e_3$$\\
$$=\left(c_1-\frac{b^2_1}{4}\right)+\left(c_2-\frac{b_1b_2}{2}\right)e_1+\left(c_3-\frac{b_1b_3}{2}\right)e_2+\left(c_4-\frac{b_1b_4}{2}\right)e_3.$$\\
Using all the above and $C=2Re(\overline{b^\prime}c^\prime)$, we find
$$C=2Re((-b_2 \cdot e_1-b_3 \cdot e_2-b_4 \cdot e_3 ) \cdot ((c_1-\frac{b^2_1}{4})+(c_2-\frac{b_1b_2}{2})e_1+(c_3-\frac{b_1b_3}{2})e_2+(c_4-\frac{b_1b_4}{2})e_3)).$$
The real part is obtained only by multiplying terms of the same kind, therefore we obtain:\\
$$C=-2b_2c_2 \alpha + b_1b_2^2 \alpha - 2b_3c_3 \beta + b_1b_3^2 \beta+2b_4c_4 \alpha\beta-b_1b_4^2\alpha\beta$$
\\
and $A=\vert b^\prime\vert^2+2Re(c^\prime)=(-\alpha b^2_2 - \beta b^3_2 + \alpha \beta b^2_4) +2 \left( c_1 -\frac{b^2_1}{4}\right)$ 

Then $A= -\alpha b^2_2 - \beta b^3_2 + \alpha \beta b^2_4 +2 c_1 -\frac{b^2_1}{2}$

Computing $B=\vert c^\prime\vert^2$  we get
$$B=\left(c_1 - \frac{b^2_1}{4} \right)^2 - \alpha \left(c_2 - \frac{b_1 b_2}{2} \right)^2 - \beta  \left(c_3 - \frac{b_1 b_3}{2} \right)^2 + \alpha \beta \left(c_4 - \frac{b_1 b_4}{2} \right)^2$$
We denote $\frac{b_1}{2} = t$
and obtain:
$$B=(c_1 - t^2 )^2 - \alpha (c_2 - t b_2 )^2 - \beta  (c_3 - t b_3  )^2 + \alpha \beta (c_4 - t b_4 )^2$$

We compute $W$ and $Y$ according to the cases of the theorem. \\
By denoting $m=|b^\prime +W|=W^2- \alpha b_2^2- \beta b_3^2+\alpha \beta_4^2$
and cu $t=b_1/2$, we apply equation  \eqref{ecuatia 9}  and we find \\
$ x_1=-t-(Wc_1-Y W- b_2 c_2 \alpha - b_3 c_3\beta +b_4 c_4 \alpha \beta -t(Wt-b_2^2 \alpha  - b_3^2 \beta + b_4^2 \alpha \beta ))/m;$\\

$ x_2=(Wc_2-b_2 c_1+b_2 Y+b_3 c_4 \beta  - b_4 c_3 \beta -tb_2 (W-t))/m;$\\

$ x_3=(Wc_3-b_2 c_4 \alpha - b_3 c_1+b_3 Y + b_4 c_2 \alpha-tb_3 (W-t))/m;$\\

$ x_4=(Wc_4-b_2 c_3+b_3 c_2 + b_4 c_1 + b_4 Y- tb_4 (W-t))/m;$\\

We obtain the solution as
\begin{equation*}
 x=x_1+x_2 e_1+x_3 e_2+x_4 e_3.
\end{equation*}
\end{proof}

\section{Numerical applications and examples} 

	For the implementation of numerical applications, let's consider the general case of $\mathbb{H}(\alpha, \beta)$, 
$b=b_1+b_2 \cdot e_1+b_3 \cdot e_2+b_4 \cdot e_3$ and 
$c=c_1+c_2 \cdot e_1+c_3 \cdot e_2+c_4 \cdot e_3$.
\newline
Using Proposition 4.1, we present the algorithm from the table
 \ref{t1}.
The algorithm described has been implemented in Scilab 6.1.1.
To verify our computations, we apply all the formulas, on some remarkable examples.
\medskip\ \vspace{1mm}
\begin{table}[h!]
\begin{tabular}{c|cc}
$Steps $ & \\ \hline
$1.$ &   & Input $\alpha, \beta, b, c$  \\ 
$ $ &   &  \\ 
$2.$ &   &  Compute C, A, B   \\
$ $ &   &  \\  
$3.$ &   & Identify case \\ 
$ $ &   &  \\ 
$4.$ & If case 1:    &  Compute $W=0$,   \\ 
$ $ &$C=0,  A\geq 4B$   & $Y=(A\pm \sqrt{A^2-4B})/2$ \\ 
$ $ &   &  \\ 
$ $ &  If case 2:     &  Compute $W=\pm \sqrt{2\sqrt{B}-A}$, \\
$ $ & $C=0,  A^2 < 4B$  & $Y=\sqrt{B}$   \\ 
$ $ &   &  \\ 
$ $ &  If case 3:   &  Solve the polynomial equations  \\

$ $ &  $C\neq 0$  & $z^3+2Az^2+(A^2-4B)z-C^2=0$  \\ 
$ $ &   & and find the positive root. \\ 
$$  &  &  \\
$$  &  &  \\
$5.$  &  &  Compute solutions using formula \eqref{ecuatia 10}. \\

\medskip
\end{tabular}
\caption{Algorithm for computing the solutions of the quadratic equation.}
\label{t1}
\end{table}

\example ([HS;02], Example 2.12) Consider the quadratic equation $x^2+x e_1 + (1+ e_2)=0,$ i.e.$, b=e_1$ and $ c=1+ e_2.$ This belongs to Case 4 in Theorem 3.3 . Then $b^\prime=e_1$ and $c^\prime=1+ e_2.$ Moreover, $A=3,$  $B=2$, $C=0.$ It is Subcase 1 in Case 4. Hence, $W=0$ and $Y=2$ or $Y=1$. Consequently, the two solutions are $ x_1= - e_1+ e_3$ and  $ x_2 = e_3.$\\

For $\alpha=-1, \beta=-1$, the solution is:\\
$ C=0.000000$\\
$ A=3.000000$\\
$ B=2.000000$\\
$ Y_1=2.000000$\\
$ Y_2=1.000000$\\
$ x_1=-0.000000-1.000000 e_1-0.000000 e_2 + 1.000000 e_3$\\
$ x_2=-0.000000-0.000000 e_1-0.000000 e_2 + 1.000000 e_3.$\\

\example ([HS;02], Example 2.13) Consider the quadratic equation $x^2+x e_1 + e_2=0,$ i.e., $b=e_1$ and $c= e_2.$ This belongs to Case 4 in Theorem 3.3 . Then $b^\prime=e_1$ and $c^\prime= e_2.$ Moreover, $A=1$, $ B=1, $ $C=0.$ It is Subcase 2 in Case 4. Hence, $W=+1$ or $-1$ and $Y= 1.$ Consequently, the two solutions are $ x_1= (e_1+1)^{-1} (1- e_2)=(1/2) (1- e_1- e_2+ e_3)$ and $x_2=(e_1-1)^{-1} (1- e_2)=(1/2) (-1- e_1+ e_2+ e_3).$ \\

For $\alpha=-1, \beta=-1$, the solution of the program:\\
$ C=0.000000$\\
$ A=1.000000$\\
$ B=1.000000$\\
$ x_1=0.500000-0.500000 e_1 - 0.500000 e_2 + 0.500000 e_3$ \\
$ x_2=-0.500000-0.500000 e_1+0.500000 e_2 + 0.500000 e_3.$\\

\example ([HS; 02], Example 2.14) Consider the quadratic equation $x^2+x e_1 + (1+e_1+ e_2)=0,$ i.e., $b=e_1$ and $ c=1+e_1+ e_2.$ This belongs to Case 4 in Theorem 3.3 . Then $b^\prime=e_1$ and $ c^\prime=1+e_1+ e_2.$ Moreover, $ A=3$,  $B=3,$ $C=2.$ It is Subcase 3 in Case 4. Now the unique positive roots of $z^3+6 z^2-3z-4$ is 1, and hence,$ W=1$ and $Y=3$ or $W=-1$ and $Y=1.$ Consequently, the two solutions are $ x_1=(1/2) (1-3 e_1- e_2+ e_3)$ and $ x_2=(1/2) (-1+ e_1+ e_2+ e_3).$\\

For $\alpha=-1, \beta=-1$, the solution of the program:\\
$ C=2.000000$\\
$ A=3.000000$\\
$ B=3.000000$\\
$ x_1=0.500000-1.500000 e_1-0.500000 e_2 + 0.500000 e_3$\\
$ x_2=-0.500000+0.500000 e_1+0.500000 e_2 + 0.500000 e_3.$\\

The results obtained in Examples 5.1-5.3 are exactly the ones obtain by direct computation by the authors in [HS;02].

In the following, we will present a few examples using the results presented above and also calculate the solutions of the equations using the described algorithm, for different values of $\alpha$ and $\beta$.

\example Next, we aim to find the solution of the equation $x^2 + b x + c = 0$ in the case where $b$ and $c$ are quaternions:\\
$b=5\cdot 1 +6 \cdot e_1+7 \cdot e_2+8 \cdot e_3$
\newline
$c=2 \cdot 1+3 \cdot e_1+4 \cdot e_2+5 \cdot e_3$

For  $\alpha=-1, \beta=-1$, we can compute
$b' = b - Re(b) = 6e_1 + 7e_2 + 8e_3$
and 
$c' = c - \frac{1}{2}Re(b)(b - \frac{1}{2}Re(b))= (2 -  \frac{25}{2} + \frac{25}{4})1 + (3 - 15) e_1 + (4 - \frac{35}{2}) e_2 + (5 - 20) e_3$ \\
$ = -\frac{17}{4} -12 e_1 - \frac{27}{2} e_2 - 15 e_3$

$A = |b'|^2 + 2 Re(c') = 6^2 + 7^2 + 8^2 + 2(-\frac{17}{4}) = 140,5$

$B = |c'|^2 = (\frac{-17}{4})^2 + 12^2 + (\frac{27}{2})^2 + (15)^2 = 569,3125$

$C = 2 Re(\overline{b'}c') = -573$

We can check that $A^2 \geq 4B$, so we can use case 4. Using the formulas in case 4, the next step is to find the values of $(W,Y)$ using one of the three situations described in the formula from case 4. Since $C \neq 0$, we will use situation 3:

$z^3 + 2Az^2 + (A^2 - 4B)z - C^2 = 0$

To find the unique positive solution $z$, we will use the Newton-Raphson method. In this case, we have:

$f(z) = z^3 + 2Az^2 + (A^2 - 4B)z - C^2$

$f'(z) = 3z^2 + 4Az + (A^2 - 4B)$

The analytical method to find the solutions of the equation is given by choosing $z_0=1$ and applying the Newton-Raphson formula. We can obtain successive values for z as the fixed number given by:
\newline
$z_1 = z_0 - \frac{f(z_0)}{f'(z_0)} = 1 - \frac{f(1)}{f'(1)}; $

$z_2 = z_1 - \frac{f(z_1)}{f'(z_1)}; $

$z_3 = z_2 - \frac{f(z_2)}{f'(z_2)}; $

$z_4 = z_3 - \frac{f(z_3)}{f'(z_3)}.$\\

Computing by this formula we use decimal fractions with many decimals, therefore we used the Scilab solver:

$$p = -328329 + 17463x + 281x^2 + x^3$$

By using of the solver in Scilab, we obtain: $W_1 =\pm 3.871934,$ and using a numerical application, we obtain:\\
$ C = -573.000000;$\\
$ A =  140.500000;$\\
$ B =  569.312500.$\\
$x_1=-0.564033 +0.008853 e_1+0.306465 e_2-0.017904 e_3$\\
$x_2=-4.435967 -5.972266 e_1-6.647896 e_2-7.945509 e_3.$\\
 
For  $\alpha=-2, \beta=-3$, the solution is\\
$C= -2295.000000$\\
$A= 594.500000$\\
$B= 2202.812500$\\
$W=\pm 3.813764$\\
$x_1=-0.593118 +0.012038 e_1+0.168839 e_2-0.004699 e_3$\\
$x_2=-4.406882 -5.982890 e_1-6.819067 e_2-7.985585 e_3$.\\

\example ([FZ, 22])

We aim to solve the following equation:\\
 $x^2 + (2+3e_1+4e_2+5e_3) x + (4-5e_1-6e_2-7e_3) = 0$

For $\alpha=-1, \beta=-1$, we write:

$(a + be_1 + c e_2 + d e_3)^2 + (2+3e_1+4e_2+5e_3)(a + be_1 + c e_2 + d e_3) + (4-5e_1-6e_2-7e_3) = 0$

We expand this equation and group the terms based on the quaternionic units:

$(a^2 - b^2 - c^2 - d^2 + 2a - 3b - 4c - 5d + 4) + (2ab + 3a + 2b - 5c + 4d -5)e_1 + (2ac + 4a + 5b + 2c -3d-6)e_2 + (2ad + 5a  -4b + 3c + 2d-7)e_3 = 0$

Thus, we can obtain a system of linear equations with 4 equations and 4 unknowns:\\
\begin{equation*}
\begin{cases}
a^2 - b^2 - c^2 - d^2 + 2a - 3b - 4c - 5d + 4 = 0\\
2ab + 3a + 2b - 5c + 4d -5 = 0\\
2ac + 4a + 5b + 2c -3d-6 = 0\\
2ad + 5a  -4b + 3c + 2d -7 = 0\\
\end{cases}
\end{equation*}

Solving this system of equations can provide us with the quaternionic solutions to the initial equation. Unfortunately, this system does not seem to have a simple and analytical solution, but we can try to solve it numerically or look for a specialized method for solving quaternionic equations.

Using the algorithm, we found the following results:\\
$ C=-248.000000$\\
$ A=56.000000$\\
$ B=317.000000$\\
$x_1=0.988335 +0.435138 e_1-0.199557 e_2+0.624407 e_3$\\
$x_2=-2.988335 -3.374360 e_1-5.198324 e_2-5.563629 e_3.$\\

For $\alpha=-2.35, \beta=-100$, the solution of the equations is \\
$x_1=1.416406 +0.030602 e_1-0.009466 e_2+0.006083 e_3$\\
$x_2=-3.416406 -2.977407 e_1-4.019286 e_2-5.005551 e_3$\\
Moreover, \\
$C= -36312.800000$\\
$A= 7502.150000$\\
$B= 43999.400000$.

\example Next, we aim to find the solution of the equation in the case where $b$ and $c$ are quaternions:
$$b=1.25 +0.2 e_1-0.31 e_2-0.69 e_3$$
and
$$c=-1 +0.56 e_1-2.35 e_2-4.56 e_2$$
Then, the equations is
$$x^2 +(1.25 +0.2 e_1-0.31 e_2-0.69 e_3) x-1 +0.56 e_1-2.35 e_2-4.56 e_2=0.$$

Using the program, for $\alpha=-1, \beta=-1$, we found the following results:\\
$C= 7.208550$\\
$A= -2.169050$\\
$B= 23.819054$\\
$W= \pm3.485216$\\
$x_1=1.117608 +-0.251329 e_1+0.667362 e_2+1.505501 e_3$\\
$x_2=-2.367608 +0.018740 e_1-0.560890 e_2-0.861963 e_3.$\\

For $\alpha=-6, \beta=-8.5$, the solution is\\
$C= 302.988862$\\
$A= 22.556700$\\
$B= 911.964612$\\
$W=\pm 7.155732$\\
$x_1=2.952866 -0.219073 e_1+0.340546 e_2+0.917961 e_3$\\
$x_2=-4.202866 -0.027102 e_1-0.234104 e_2-0.235706 e_3.$\\
 
\example Next, we aim to calculate by using of the program an example where $C=0$:

Find the solutions of the equation: $x^2 + (e_1+e_2+e_3) x + (-3e_1-4e_2+7e_3) = 0$.

We can see that $b=e_1+e_2+e_3 \notin \mathbb{R}$, so we need to use the formula from case 4.

Firstly, we will calculate the values of $b^\prime$, $c^\prime$, $A$, $B$, and $C$:

$b^\prime = b - \operatorname{Re}(b) =e_1+e_2+e_3$

$c^\prime = c - \frac{\operatorname{Re}(b)}{2}(b - \frac{\operatorname{Re}(b)}{2}) = -3e_1-4e_2+7e_3$

$A = |b^\prime|^2 + 2\operatorname{Re}(c^\prime) =3$

$B = |c^\prime|^2 = 74$

$C = 2\operatorname{Re}(\overline{b^\prime}c^\prime) = 0$

The next step is to find the values of $(W,Y)$ using one of the three situations described in the formula from case 4. Since $C=0$ and $A^2 < 4B$.
Now we can calculate $(W,Y)$:
$W=\pm \sqrt{2\sqrt{B}-A}=\pm 3,7689057476$ and
 $Y=\sqrt{B}=8,602325267$.

By using of the program, we have found the following results:\\
$ C=0.000000$\\
$ A=3.000000$\\
$ B=74.000000$\\
$ W=\pm 3.768906$\\
$ Y= 8.602325$\\
$ x_1=1.884453+0.796552 e_1+0.608748 e_2 -2.091566 e_3$\\
$ x_2=-1.884453-0.517828 e_1-1.143758 e_2 +0.975319 e_3.$\\

The same equation can be solved for $\alpha=-6$ and $\beta=-9$. In this case, $C\neq0$. We get\\
$C= 2088.000000$\\
$A= 528.000000$\\
$B= 2844.000000$\\
$W=\pm 3.919010$\\
$x_1=1.959505 -0.537980 e_1-1.973780 e_2-3.017625 e_3$\\
$x_2=-1.959505 +0.399290 e_1-0.070390 e_2+0.024986 e_3$.\\

\example Next, we intend to use the program to calculate an example where C=0:

Let's find the solutions of the equation: $x^2 + (e_1+e_2+e_3) x + (-e_1+e_3) = 0$.

We can see that $b=e_1+e_2+e_3 \notin \mathbb{R}$, so we need to use the formula from case 4.

Firstly, we will calculate the values of $b^\prime$, $c^\prime$, $A$, $B$ and $C$:

$b^\prime = b - \operatorname{Re}(b) =e_1+e_2+e_3$

$c^\prime = c - \frac{\operatorname{Re}(b)}{2}(b - \frac{\operatorname{Re}(b)}{2}) = -e_1+e_3$

$A = |b^\prime|^2 + 2\operatorname{Re}(c^\prime) =3$

$B = |c^\prime|^2 = 2$
$C = 2\operatorname{Re}(\overline{b^\prime}c^\prime) = 0$

The next step is to find the values of $(W,Y)$ using one of the three situations described in the formula of case 4. Since  $C=0$ and $A^2 \geq4B$, we will use situation 1:

$W=0,$ 
$Y=(A\pm \sqrt{A^2-4B})/2$  result $Y_1=2$,$Y_2= 1$ 

Calculating with the numerical application, we get:\\
$ C=0.000000$\\
$ A=3.000000$\\
$ B=2.000000$\\
$ Y_1=2.000000$\\
$ Y_2=1.000000$\\
$ x_1=-0.000000-0.333333 e_1-0.666667 e_2 - 0.333333 e_3$\\
$ x_2=-0.000000-0.000000 e_1-0.333333 e_2 - 0.000000 e_3.$\\

For $\alpha=-100, \beta=-100$, we get $C\neq0$, like in the other example, and the solution is\\
$C= 19800.000000$\\
$A= 10200.000000$\\
$B= 10100.000000$\\
$W=\pm 1.940836$\\
$x_1=0.970418 -0.989912 e_1-0.999903 e_2-0.999995 e_3$\\
$x_2=-0.970418 +0.009513 e_1-0.000097 e_2+0.000191 e_3.$\\

\example ([FZ; 22]) Let $f_n$ be the Fibonacci sequence define as $f_0=0, f_1=1$ and $f_k=f_{k-1}+f_{k-2}$. We define the quaternion $F_n=f_n+f_{n+1}e_1+f_{n+2}e_2+f_{n+3}e_3$.

Consider the monic quadratic equation $x^2+F_n x+F_m=0$. We use the same algorithm for solving the equation.

For $n=3,m=3$, case discussed in ([FZ; 22]),  we obtain
$$F_3=2 +3 e_1+5 e_2+8 e_3$$
and the equation 
$x^2 + (2 +3 e_1+5 e_2+8 e_3) x + (2 +3 e_1+5 e_2+8 e_3) = 0$.

Solving the equations for $\alpha=-1,\beta=-1$, and we get\\
$C= 0.000000$\\
$A= 100.000000$\\
$B= 1.000000$\\
$Y_1= 99.989999$\\
$Y_2= 0.010001$\\
$x_1=-1.000000 -3.030306 e_1-4.560714 e_2-8.080816 e_3$\\
$x_2=-1.000000 +0.030306 e_1+0.540306 e_2+0.080816 e_3.$\\

Solving the equations for $\alpha=-6.3,\beta=-5.25$, and we get\\
$C= 0.000000$\\
$A= 2306.750000$\\
$B= 1.000000$\\
$Y_1= 2306.749566$\\
$Y_2= 0.000434$\\
$x_1=-1.000000 +-3.001301 e_1-4.870961 e_2-8.003470 e_3$\\
$x_2=-1.000000 +0.001301 e_1+0.133376 e_2+0.003470 e_3.$\\

For $n=5, m=10$ we obtain 
$$F_5=5+8 e_1+13 e_2+21e_3,$$
$$F_{10}=55 +89 e_1+144 e_2+233 e_3.$$
and the equation $x^2+F_ 5x+F_{10}=0$.
Then, the solution for $\alpha=-1,\beta=-1$ found by the algorithm is\\
$C= 11584.000000$\\
$A= 771.500000$\\
$B= 52150.062500$\\
$W=\pm 13.722364$\\
$x_1=4.361182 -9.008123 e_1-10.308573 e_2-23.657396 e_3$\\
$x_2=-9.361182 +1.019720 e_1+5.966780 e_2+2.645800 e_3$.\\

For $\alpha=-6.3,\beta=-5.25$ the solution provided by the algorithm is\\
$C= -272916.525000$\\
$A= 15974.025000$\\
$B= 1175231.943750$\\
$W=\pm 16.934907$\\
$x_1=5.967453 -8.058866 e_1-11.642625 e_2-21.158442 e_3$\\
$x_2=-10.967453 +0.062114 e_1+1.552659 e_2+0.157823 e_3.$\\

\example Let $p_n$ be the Pell sequence define as $p_0=0, p_1=1$ and $p_k=2p_{k-1}+p_{k-2}$. Consider the quaternions $P_n=p_n+p_{n+1}e_1+p_{n+2}e_2+p_{n+3}e_3$. 
We solve the monic quadratic equation $x^2 +P_n x+P_m=0$.

For $n=3,m=3$, we get
$P_3=3 +7 e_1+17 e_2+41 e_3$ and the equation is $x^2 +(3 +7 e_1+17 e_2+41 e_3)x+3 +7 e_1+17 e_2+41 e_3=0$. 

Solving the equations for $\alpha=-1,\beta=-1$ using the algorithm we obtain\\
$C= -2019.000000$\\
$A= 2020.500000$\\
$B= 505.312500$\\
$W= \pm 0.999011$\\
$x_1=-1.000494 +0.003464 e_1+0.292570 e_2+0.020287 e_3$\\
$x_2=-1.999506 -7.003464 e_1-16.724253 e_2-41.020287 e_3$.\\

For $\alpha=-7,\beta=-6$ the solutions are\\
$C=-72679.000000$\\
$A= 72680.500000$\\
$B= 18170.312500$\\
$W=\pm 0.999972$\\
$x_1=-1.000014 +0.000096 e_1+0.055517 e_2+0.000564 e_3$\\
$x_2=-1.999986 -7.000096 e_1-16.944950 e_2-41.000564 e_3.$\\

For $n=12,m=19$, the quaternions are
$P_{12}=8119 +19601 e_1+47321e_2+114243 e_3$ and 
$P_{19}=3880899 +9369319 e_1+22619537 e_2+54608393 e_3$.
The equations is $x^2 +P_{12} x+P_{19}=0$.
Solving for $\alpha=-1,\beta=-1$, we get\\
$C= -112279524556439.000000$\\
$A= 15649742008.500000$\\
$B= 201223166914529952.000000$\\
$W=\pm 7162.787683$\\
$x_1=-478.106158 +0.284778 e_1+136.813987 e_2+1.659808 e_3$\\
$x_2=-7640.893842 -19601.284778 e_1-47185.561043 e_2-114244.659808 e_3.$\\

For $\alpha=-7,\beta=-6$ the solutions are\\
$C=-4041981872234103.000000$\\
$A= 564261307428.500000$\\
$B= 7238333535486963712.000000$\\
$W=\pm 7162.990016$\\
$x_1=-478.004992 +0.007936 e_1+26.572949 e_2+0.046254 e_3$\\
$x_2=-7640.995008 -19601.007936 e_1-47294.465369 e_2-114243.046254 e_3.$

\example Consider now the Lucas number sequences define as $l_0=2, l_1=1$ and $l_n=l_{n-1}+l_{n-2}$. We define the quaternion $L_n=l_n+l_{n+1}e_1+l_{n+2}e_2+l_{n+3}e_3$. We solve the monic quadratic equation $x^2 +L_n x+L_m=0$.

For $n=3, m=8$, the quaternions are
$L_3=4 +7 e_1+11 e_2+18 e_3$ and
$L_8=47 +76 e_1+123 e_2+199 e_3$.

Solving the equation $x^2 +L_3 x+L_8=0$ for $\alpha=-1,\beta=-1$, we get\\
$C= 8958.000000$\\
$A= 580.000000$\\
$B= 42463.000000$\\
$W=\pm 13.777285$\\
$x_1=4.888642 -8.113676 e_1-8.726917 e_2-20.805123 e_3$\\
$x_2=-8.888642 +1.040556 e_1+5.802217 e_2+2.878243 e_3$.\\

On the other hand, for $\alpha=-3, \beta=-10$, the solution are\\
$C=200864.000000$\\
$A= 11163.000000$\\
$B= 912461.000000$\\
$W=\pm 17.747518$\\
$x_1=6.873759 -7.095766 e_1-10.394477 e_2-18.193193 e_3$\\
$x_2=-10.873759 +0.051875 e_1+0.848658 e_2+0.197582 e_3.$\\

For $n=11, m=14$, $L_{11}=199 +322e_1+521e_2+843e_3$ and 
$L_{14}=843+1364e_1+2207e_2+3571e_3$. 

Solving the equation $x^2 +L_{11} x+L_{14}=0$ for $\alpha=-1, \beta=-1$, we get\\
$C= -4638388100.000000$\\
$A= 24326817.500000$\\
$B= 221017587902.562500$\\
$W=\pm 190.527592$\\
$x_1=-4.236204 +0.000233 e_1+0.283353 e_2+0.000622 e_3$\\
$x_2=-194.763796 -322.000241 e_1 -520.717414 e_2 -843.000621 e_3.$\\

Finally, we solve the same equation for $\alpha=-1.236, \beta=-10.023$, the solution are\\
$C= -2220150838.889460$\\
$A= 11634516.036772$\\
$B= 105832184609.751312$\\
$W=\pm 190.527281$\\
$x_1=-4.236359 +0.000487 e_1+0.243975 e_2+0.001297 e_3$\\
$x_2=-194.763641 -322.000504 e_1-520.757626 e_2 -843.001295 e_3.$\\

\subsection*{Conclusion}

	In this article, we have provided an algorithm in Scilab which allows us to find solutions for the monic quadratic equation $x^2+bx+c=0$, with $b,c \in 
\mathbb{H}(\alpha, \beta)$.

	In Theorem \ref{t23}, the authors offer solutions for all cases of the monic equation $x^2+bx+c=0$. We are interested only in cases $3$ and $4$ of the theorem.
The article presents several equations solved using the algorithm, implemented in Scilab.

	By assigning specific values to the two quaternions, b and c, in the form of $b = b_1 + b_2e_1 + b_3e_2 + b_4e_3$ and $c = c_1 + c_2e_1 + c_3e_2 + c_4e_3$, and utilizing the formulas provided in the article, we perform the following calculations:
Compute the values of $A, B$, and $C$:
A is determined by evaluating the expression $A = |b'|^2 + 2Re(c')$,
where $b' = b - Re(b)$ and $c' = c - (Re(b)/2)(b - (Re(b))/2)$.
$B$ is computed as $B = |c'|^2$.
$C$ is obtained by calculating $C=2Re(\overline{b^\prime}c^\prime)$.
Identify the case we are in, based on the four cases specified in the theorem.
Proceeding with the determined case, we find the two solutions of the monic quadratic equation, $x^2 + bx + c = 0$, using the appropriate formulas presented in the article.
	That this detailed procedure allows us to obtain precise and accurate solutions for the given quadratic equation in the context of the algebra of real quaternions.

	The algorithm can solve monic quadratic equations for any base that respects the multiplication table of quaternions.

\subsection*{Acknowledgment}
We would like to thank the reviewers for their constructive feedback, expertise, and valuable comments that have significantly improved this article.

Geanina ZAHARIA,\\
PhD student at Doctoral School of Mathematics,\\
Ovidius University of Constanta, Romania,\\
geanina.zaharia@365.univ-ovidius.ro\\
geaninazaharia@yahoo.com\\

Diana-Rodica MUNTEANU\\
Ovidius University of Constanța,\\ 
Faculty of Psychology and
Educational Sciences\\
diana\_munteanu@365.univ-ovidius.ro\\ diana.rodica.merlusca@gmail.com\\

\end{document}